\documentclass[10pt,reqno]{amsart}

\usepackage{amssymb,latexsym,amsmath}

\def\sup{\mathop{\rm sup}}
\def\Real{\mathop{\rm Re}}

\newtheorem{theorem}{Theorem}[section]

\newtheorem{corrollary}[theorem]{Corrollary}
\newtheorem{remark}[theorem]{Remark}
\begin{document}

\title[]{Stability results of some abstract evolution equations}

\author{N. S. Hoang}

\address{Mathematics Department, University of West Georgia,
Carrollton, GA 30116, USA}
\email{nhoang@westga.edu}


\date{}


\begin{abstract}
The stability of the solution to the equation $\dot{u} = A(t)u + G(t,u)+f(t)$, $t\ge 0$, $u(0)=u_0$ is studied. 
Here $A(t)$ is a linear operator in a Hilbert space $H$ and $G(t,u)$ is a nonlinear operator in $H$ for any fixed $t\ge 0$. We assume that $\|G(t,u)\|\le \alpha(t)\|u\|^p$, $p>1$, and the spectrum of $A(t)$ lies in the half-plane $\Real \lambda \le \gamma(t)$ where $\gamma(t)$ can take positive and negative values. We proved that the equilibrium solution $u=0$ to the equation is Lyapunov stable under 
persistantly acting 
perturbations $f(t)$ if $\sup_{t\ge 0}\int_0^t \gamma(\xi)\, d\xi <\infty$ and $\int_0^\infty \alpha(\xi)\, d\xi<\infty$. In addition, if $\int_0^t \gamma(\xi)\, d\xi \to -\infty$ as $t\to\infty$, then we proved that the equilibrium solution $u=0$ is asymptotically stable under 
persistantly acting 
perturbations $f(t)$. Sufficient conditions for the solution $u(t)$ to be bounded and for $\lim_{t\to\infty}u(t) = 0$ are proposed and justified.
\end{abstract}
\maketitle

\pagestyle{plain}

\section{Introduction}
Consider the equation
\begin{equation}
\label{eq1}
\dot{u} = A(t)u + G(t,u) + f(t),\quad t\ge 0,\quad u(0)= u_0,\quad \dot{u} := \frac{d u}{dt}.
\end{equation}
Here, $u(t)$ is a function of $t\ge 0$ with values in a Hilbert space $H$, $A(t):H\to H$ is a linear, closed, and densely defined operator in $H$, 
\begin{equation}
\label{ae4}
\Real \langle u, A(t)u\rangle \le \gamma(t)\|u\|^2, \qquad t\ge 0,
\end{equation}
$G(t,u)$ is a nonlinear operator in $H$ for any fixed $t\ge 0$,
\begin{equation}
\label{ae5}
\|G(t,u)\| \le \alpha(t) \|u\|^p,\qquad p>1,\quad t\ge 0,
\end{equation}
and $f(t)$ is a function on $\mathbb{R}_+=[0,\infty)$ with values in $H$,
\begin{equation}
\label{ae6}
\|f(t)\| \le \beta(t),\qquad t\ge 0.
\end{equation}
Note that inequality \eqref{ae5} implies that $G(t,0) = 0$. Thus, $u=0$ is an equilibrium solution to the equation 
$$
\dot{u} = A(t)u + G(t,u),\qquad t\ge 0.
$$
It is assumed that $\alpha(t)$, $\beta(t)$, and $\gamma(t)$ in inequalities \eqref{ae4}--\eqref{ae6} are in $L^1_{loc}([0,\infty))$ and that $\alpha(t)$ and $\beta(t)$ are nonnegative on $[0,\infty)$.  
%We assume that $G(t,u)$ is 
%Lipschitz continuous with respect to $u$ in the ball $B(u_0,R)=\{u\in H:\|u-u_0\|\le R\}$ and is continuous with respect to $t$ on $\mathbb{R}_+$ in the operator norm. 
Also, we assume that equation \eqref{eq1} has a unique local solution. 
A stronger assumption on the local existence of equation \eqref{eq1} is made in Assumption A below. By a solution to problem \eqref{eq1} we mean a classical solution. 
Specifically, a global solution to \eqref{eq1} is a continuous differentiable function $u:[0,\infty)\to H$ which satisfies equation \eqref{eq1}. A local solution to equation \eqref{eq1} 
is a continuous differentiable function $u:[0,T)\to H$, for some $T>0$, which solves equation \eqref{eq1}. Thus, the solution space for global existence is $C^1([0,\infty);H)$ and for local existence is $C^1([0,T);H)$. 
Recall that a local solution to problem \eqref{eq1} exists and is unique if $A(t)$ is a generator of a $C_0$-semigroup. 

Take inner product of both sides of equation \eqref{eq1} with $u$ to get
$$
\langle u, \dot{u}\rangle = \langle u, A(t)u \rangle + \langle u, G(t,u)\rangle + \langle u, f(t)\rangle,\qquad t\ge 0. 
$$
%Note that this equality holds true for any $u(t)$ which solves equation \eqref{eq1}. 
Denote $g(t) := \|u(t)\|$, take the real part of the equation above, and use the triangle inequality to get
\begin{equation*}
\begin{split}
\dot{g}(t)g(t) &\le \Real \langle u, A(t)u \rangle + |\langle u, G(t,u)\rangle| + \beta(t)g(t),\\
&\le \gamma(t)g^2(t) + \alpha(t)g^{p+1}(t) + \beta(t)g(t), \qquad t\ge 0.
\end{split}
\end{equation*}
This implies
\begin{equation}
\label{eq2}
\dot{g} \le \gamma(t)g(t) + \alpha(t)g^p(t) + \beta(t),\qquad t\ge 0,\quad g(0) = \|u_0\|.
\end{equation}
Note that in inequality \eqref{eq2} the functions $\alpha(t)$ and $\beta(t)$
are non negative on $\mathbb{R}_+$.

The stability of solutions to equation \eqref{eq1} has been studied in the literature (see, e.g., \cite{C}, \cite{CL}, \cite{D}, and \cite{L}). Stability of solutions of abstract equations in Banach and Hilbert spaces was studied in \cite{DK}, \cite{H}, and \cite{T}. In \cite{R1} stability of solutions of abstract equations in Hilbert spaces was studied using nonlinear inequalities. In \cite{R2}--\cite{R5} stability of the solution to equation \eqref{eq1} was studied using nonlinear inequalities under the assumption that the spectrum of $A(t)$ lie in the half-plane $\Real \lambda \le \gamma(t)$ where $0>\gamma(t)\to 0$ as $t\to\infty$ (see \cite{R2} and \cite{R3}) or $0<\gamma(t)\to 0$ as $t\to\infty$ (see \cite{R4}). 
In \cite{R6} stability of solutions to abstract evolution equations with delay was studied. 

The classical stability result of equation \eqref{eq1} states that if $A(t) \equiv A$ a constant matrix whose eigenvalues lie in the half-plane 
$\Real \lambda < \sigma_0 < 0$, and $\alpha(t)$ and $f(t)$ are identically equal to zero, then the solution to problem \eqref{eq1} exists globally, is unique, and is asymptotically stable. If the matrix $A$ has an eigenvalue in the half-plane 
$\Real(\lambda)>0$, then, in general, $\lim_{t\to\infty}u(t) = \infty$. 

In this paper we study the stability of the solution to equation \eqref{eq1} under a more relaxed condition on the spectrum of $A(t)$ than those used in the literature. Namely, we allow the spectrum of $A(t)$ to lie in the half-plane $\Real (\lambda) \le \gamma(t)$, where $\gamma(t)$ can take positive and negative values. In \cite{R2}--\cite{R4} it was assumed either $\gamma(t) >0$ or $\gamma(t) < 0$ on $\mathbb{R}_+$. 
We give sufficient conditions on the functions $\alpha(t)$, $\beta(t)$, and $\gamma(t)$ which yield stability properties of the solution to equation \eqref{eq1}. 

The novelty of the stability results in this paper compared to those in \cite{R2}--\cite{R5} is: Our results do not require to find a function $\mu(t)>0$ which solves a nonlinear inequality as those in \cite{R2}--\cite{R4}. In particular, our results are applicable for the case when $\gamma(t) = \sin t$ (or  $\gamma(t) = \frac{\sin t }{(t+1)^a}, 0<a<1$) and $\alpha(t)$ is a measure, positive, and intergrable function on $\mathbb{R}_+$.  These cases are not easy to treat using the results in \cite{R2}--\cite{R4} as it is not easy to find functions $\mu(t)$ which solve nonlinear inequalities in \cite{R2}--\cite{R4} for general $\gamma(t)$, $\alpha(t)$, and $\beta(t)$.    The conditions on $\alpha(t)$ and $\gamma(t)$ in Theorem \ref{thm2.1} in this paper are also more relaxed than those in Theorem 2 in \cite{R4}. 
Specifically, in Theorem \ref{thm2.1} we proved that if $\sup_{t\ge 0}\int_0^t \gamma(\xi)\, d\xi <\infty$ and $\int_0^\infty \alpha(\xi)\, d\xi < \infty$, then the equilibrium solution to problem \eqref{eq1} is Lyapunov stable under persistently acting perturbations. 
In Theorem 2 in \cite{R4} it is required that $\gamma(t)>0$ and that $\int_0^\infty [\gamma(\xi) + \alpha(\xi)]\,d\xi$ is not `large' to get the same stability. 
Other results in this paper are Theorem \ref{thm2.3} and Theorem \ref{thm2.7} in which we give sufficient conditions for the solution to problem \eqref{eq1} to be asymptotically stable. The rate of decay of the solution to problem \eqref{eq1} of exponential type is given in Theorem \ref{thm2.3} and Corrollary \ref{corrollary2.5}. 

Throughout the paper, we assume that the following assumption holds.

\vskip 5pt
{\bf Assumption A.} The equation 
$$
\dot{u} = A(t)u + G(t,u) + f(t),\quad t\ge t_0,\quad u(t_0)= \tilde{u}_0,\quad \dot{u} := \frac{d u}{dt}
$$
where $A(t)$, $G(t,u)$, and $f(t)$ are defined as earlier 
has a unique local solution for any $t_0\ge 0$ and $\tilde{u}_0\in H$. 

\section{Main results}

\begin{theorem}
\label{thm2.1}
Assume that 
\begin{equation}
\label{eq3.1}
M :=\sup_{t\ge 0}\int_0^t \gamma(\xi)\, d\xi  <\infty,\qquad \int_0^\infty \alpha(t)\, dt <\infty.
\end{equation}
Then the equilibrium solution $u=0$ to problem \eqref{eq1} is Lyapunov stable under persistently acting perturbations $f(t)$. 
\end{theorem}

\begin{remark}{\rm 
The term $f(t)$ in equation \eqref{eq1} is called {\it persistently acting perturbations}.
{\it `Stable under persistently acting perturbations $f(t)$'} means that given any $\epsilon>0$ arbitrarily small, if $\|f(t)\|$ is sufficiently small, then there exists $\delta>0$ such that if $\|u(0)\| <\delta$ then $\|u(t)\|<\epsilon$ for all $t\ge 0$. 

The first condition in \eqref{eq3.1} is necessary for the solution to equation \eqref{eq1} to be bounded, in general. Indeed, if the first condition in \eqref{eq3.1} does not hold, then the function $v(t) := u_0e^{\int_0^t \gamma(\xi)\, d\xi}$ is unbounded and solves the equation $\dot{u} = \gamma(t)u$, $t\ge 0$, $u(0)=u_0\not=0$. This initial value problem is a special case of equation \eqref{eq1} when $A(t)u = \gamma(t)u$, $G(t,u)\equiv 0$, and $f(t)\equiv 0$. 
}
\end{remark}

\begin{proof}[Proof of Theorem \ref{thm2.1}]
Let $\epsilon>0$ be arbitrarily small. Define 
\begin{equation}
\mu(t) := e^{-\int_0^t [\gamma(\xi) + \epsilon^{p-1}\alpha(\xi)]\, d\xi},\qquad t\ge 0. 
\end{equation}
Then
\begin{equation}
\label{eqs1}
\mu(t) \ge e^{-M_1},\quad t\ge 0,\qquad M_1:= M + \epsilon^{p-1}\int_0^\infty \alpha(\xi)\, d\xi.
\end{equation}
Choose $\delta>0$ sufficiently small such that
\begin{equation}
\label{eq5}
\delta e^{M_1} < \frac{\epsilon}{3}.
\end{equation}

{\it Let us prove that if $0\le g(0)=\|u_0\|<\delta$ and $\beta(t) = \|f(t)\|$ is sufficiently small, then $\|u(t)\|<\epsilon$ for all $t\ge 0$.} 

Let $T>0$ be the largest real value such that
\begin{equation}
\label{eqa1}
g(t) = \|u(t)\| \le \epsilon,\qquad \forall t\in [0,T].
\end{equation}
We claim that $T=\infty$. Assume the contrary. 
Thus, $T$ is finite and, by the continuity of $g(t)$,
\begin{equation}
\label{eqs2}
g(T) = \|u(T)\| = \epsilon.
\end{equation}
Choose $f(t)$ such that the function $\beta(t) = \|f(t)\|$ satisfies the inequality 
\begin{equation}
\label{eq6}
\frac{\int_0^t \beta(\xi)\mu(\xi)\, d\xi}{\mu(t)} < \frac{\epsilon}{3},\qquad\mu(t) = e^{-\int_0^t (\gamma(\xi) + \epsilon^{p-1}\alpha(\xi))\, d\xi}.
\end{equation}
Inequality \eqref{eq6} holds true if $\|f(t)\|=\beta(t)$ is sufficiently small. 
It follows from inequalities \eqref{eq2} and \eqref{eqa1} that
$$
\dot{g} \le \gamma(t)g(t) + \alpha(t)\epsilon^{p-1}g(t) + \beta(t),\qquad 0\le t\le T.
$$
This implies
\begin{equation}
\label{eqa2}
\begin{split}
\frac{d}{dt}\big(g(t)\mu(t)\big) \le \beta(t)\mu(t),\qquad \mu(t) = e^{-\int_0^t (\gamma(\xi) + \epsilon^{p-1}\alpha(\xi))\,d\xi},\qquad 0\le t\le T.
\end{split}
\end{equation}
Integrate inequality \eqref{eqa2} from $0$ to $t$ to get
$$
g(t)\mu(t) - g(0)\mu(0) \le \int_0^t \beta(\xi)\mu(\xi)\, d\xi,\qquad 0\le t\le T.
$$
This, inequality \eqref{eqs1}, and inequality \eqref{eq6} imply
\begin{equation}
\label{eq8}
\begin{split}
g(t) \le \frac{g(0)}{\mu(t)} + \frac{\int_0^t \beta(\xi)\mu(\xi)\,d\xi}{\mu(t)} \le g(0)e^{M_1} + \frac{\epsilon}{3},\qquad \forall t\in [0,T].
\end{split}
\end{equation}
It follows from inequalities \eqref{eq5} and \eqref{eq8} and the inequality $g(0)<\delta$ that
\begin{equation}
\label{}
g(t) \le \delta e^{M_1} + \frac{\epsilon}{3} \le \frac{2\epsilon}{3},\qquad \forall t\in [0,T].
\end{equation}
This implies $g(T) \le \frac{2\epsilon}{3}$ which contradicts to relation \eqref{eqs2}. 
%Hence, there exists $\theta>0$ such that $g(t) = <\epsilon$, $\forall t\in [0,T+\theta]$.
This contradiction implies that $T=\infty$, i.e.,
$$
\|u(t)\| = g(t) \le \epsilon,\qquad \forall t\ge 0.
$$
Thus, the equilibrium solution $u=0$ is Lyapunov stable under persistently acting perturbations $f(t)$. 
Theorem \ref{thm2.1} is proved. 
\end{proof}

\begin{theorem}
\label{thm2.3}
Assume that 
\begin{equation}
\label{eq10}
M := \sup_{t\ge 0}\int_0^t \gamma(\xi)\, d\xi  < \infty,
%\qquad  \int_0^\infty \alpha(t)\, dt <\infty,
\end{equation}
\begin{equation}
\label{eq11}
\frac{1}{(g(0) + \omega)^{p-1}} > (p-1)\int_0^\infty \frac{\alpha(\xi)}{\nu^{p-1}(\xi)}\,d\xi,\qquad \omega = const>0,\quad 
\nu(t) = e^{-\int_0^t \gamma(\xi)\, d\xi}. 
\end{equation}
If $\beta(t)=\|f(t)\|$ satisfies the inequality
\begin{equation}
\label{eqs6}
\frac{\beta(t)\nu^p(t)}{\alpha(t)} \le \omega^p,\qquad t\ge 0, 
\end{equation}
then the solution $u(t)$ to problem \eqref{eq1} exists globally, is bounded, and satisfies
\begin{equation}
\label{eqa7}
\|u(t)\| \le C_2 e^{\int_0^t \gamma(\xi)\, d\xi},\qquad t\ge 0,\quad C_2 = const>0.
\end{equation}
Moreover, if
\begin{equation}
\label{eq12.0}
\lim_{t\to\infty}\int_0^t \gamma(\xi)\, d\xi = -\infty,
%\qquad  \int_0^\infty \alpha(t)\, dt <\infty,
\end{equation}
then
\begin{equation}
\label{eq12}
\lim_{t\to\infty} u(t) = 0.
\end{equation}
\end{theorem}

\begin{remark}{\rm
Inequality \eqref{eq11} is a natural assumption. If inequality \eqref{eq11} does not hold for any $\omega \ge 0$, then the solution $g(t)$ to inequality \eqref{eq2} may blow up at a finite time even for the case when $f(t)\equiv 0$. For example, one can verify that the solution to the equation
$$
\dot{g} = \gamma(t)g(t) + \alpha(t)g^p(t),\qquad t\ge 0,\qquad g(0) = g_0,
$$
is
$$
g(t)=\tilde{g}(t) := \frac{1}{\nu(t)}\bigg(\frac{1}{g^{1-p}_0 - (p-1) \int_0^t \frac{\alpha(\xi)}{\nu^{p-1}(\xi)}\, d\xi}\bigg)^{\frac{1}{p-1}}.
$$
The function $\tilde{g}(t)$ blows up at a finite time $t=t_0$ if $t_0$ is the solution to the equation 
$$
0=\frac{1}{g^{p-1}_0} - (p-1) \int_0^t \frac{\alpha(\xi)}{\nu^{p-1}(\xi)}\, d\xi.
$$
This equation has a solution $t_0>0$ if
$$
\frac{1}{g^{p-1}_0} < (p-1) \int_0^\infty \frac{\alpha(\xi)}{\nu^{p-1}(\xi)}\, d\xi.
$$
If $g(t)$ blows up at a finite time, then the solution $u(t)$ to equation \eqref{eq1} blows up at a finite time as well due to the relation $g(t) = \|u(t)\|$. 

Inequality \eqref{eqs6} holds if $\beta(t) = \|f(t)\|$ is sufficiently small. 
Relation \eqref{eq12} implies, under assumptions \eqref{eq11} and \eqref{eq12.0}, that the equilibrium solution $u=0$ to problem \eqref{eq1} is asymptotically stable under persistantly acting perturbations $f(t)$. 
}
\end{remark}

\begin{proof}[Proof of Theorem \ref{thm2.3}]
{\it Let us first show that the solution $u(t)$ to problem \eqref{eq1} exists globally.} Assume the contrary. Thus, there exists a finite number $T>0$ such that the maximal interval of existence of $u(t)$ is $[0,T)$. 
Inequality \eqref{eq2} is equivalent to 
\begin{equation}
\label{eq3}
\frac{d}{dt}\big(g(t)\nu(t)\big) \le \alpha(t)g^p(t)\nu(t) + \beta(t)\nu(t),\qquad  0\le t< T,\quad \nu(t) := e^{-\int_0^t \gamma(\xi)\, d\xi}.
\end{equation}
Inequalities \eqref{eq3} and \eqref{eqs6} imply
\begin{equation}
\label{eqa4}
\begin{split}
\frac{d}{dt}\big(g(t)\nu(t)\big) &\le \frac{\alpha(t)}{\nu^{p-1}(t)}(g(t)\nu(t))^p + \beta(t)\nu(t)\\
&= \frac{\alpha(t)}{\nu^{p-1}(t)}\bigg[ (g(t)\nu(t))^p + \frac{\beta(t)\nu^p(t)}{\alpha(t)} \bigg] \\
&\le \frac{\alpha(t)}{\nu^{p-1}(t)}\bigg[ (g(t)\nu(t))^p + \omega^p \bigg]\\
&\le \frac{\alpha(t)}{\nu^{p-1}(t)}\bigg( g(t)\nu(t)+ \omega\bigg)^p,\qquad 0\le t< T,\qquad p>1. 
\end{split}
\end{equation}
Here we have used the inequality $a^p + b^p \le (a+b)^p$, $a,b\ge 0$, $p>1$. 
Inequality \eqref{eqa4} can be rewritten as
$$
\frac{d}{dt}\bigg( \frac{\big[g(t)\nu(t)+ \omega\big]^{1-p} }{1-p} \bigg) \le \frac{\alpha(t)}{\nu^{p-1}(t)},\qquad 0\le t< T.
$$
Integrate this inequality from 0 to $t$ to get
\begin{equation}
\label{}
\frac{\big[g(t)\nu(t) + \omega\big]^{1-p} - \big[g(0) + \omega\big]^{1-p}}{1-p}  
\le \int_0^t \frac{\alpha(\xi)}{\nu^{p-1}(\xi)}\, d\xi,\qquad 0\le t< T. 
\end{equation}
Therefore,
\begin{equation}
\label{eqx20}
\big[g(t)\nu(t) + \omega\big]^{p-1}  \le \frac{1}{(g(0) + \omega)^{1-p} - (p-1) \int_0^t \frac{\alpha(\xi)}{\nu^{p-1}(\xi)}\, d\xi},\qquad 0\le t<T. 
\end{equation}
Inequality \eqref{eq11} implies that the right-hand side of \eqref{eqx20} is well-defined for all $t\ge 0$. 
Thus, from \eqref{eqx20} one gets
\begin{equation}
\label{eq18}
\big[g(t)\nu(t) + \omega\big]^{p-1}  \le \frac{1}{(g(0) + \omega)^{1-p} - (p-1) \int_0^\infty \frac{\alpha(\xi)}{\nu^{p-1}(\xi)}\, d\xi}:=M_3,\qquad 0\le t< T. 
\end{equation}
It follows from relation \eqref{eq10} that 
$$
\nu(t) = e^{-\int_0^t \gamma(\xi)\, d\xi} \ge e^{-M},\qquad 0\le t< T. 
$$ 
This and inequality \eqref{eq18} imply that
\begin{equation}
\label{eq19}
g(t) \le \frac{M_3^{\frac{1}{p-1}} - \omega}{\nu(t)} \le e^{M} (M_3^{\frac{1}{p-1}} - \omega),\qquad 0\le t< T. 
\end{equation}
This and the continuity of $u(t)$ imply that $\|u(T)\|$ is finite and $u(t)$ exists on $[0,T]$. This and Assumption A imply that the existence of the solution $u(t)$ to equation \eqref{eq1} can be extended to a larger interval, namely, $[0,T+\delta)$ for some $\delta>0$. This contradicts the definition of $T$. The contradiction implies that $T=\infty$, i.e., $u(t)$ exists globally. 
The boundedness of $u(t)$ follows directly from inequality \eqref{eq19} with $T=\infty$. 

{\it Let us prove \eqref{eq12} assuming that \eqref{eq12.0} holds.}
Let $C_2 := M_3^{\frac{1}{p-1}} - \omega$. Then inequality \eqref{eqa7} follows from the first inequality in \eqref{eq19} and the relations $g(t) = \|u(t)\|$ and $\nu(t) = e^{-\int_0^t \gamma(\xi)\,d\xi}$. 
If relation \eqref{eq12.0} holds, then $\nu(t)=e^{-\int_0^t \gamma(\xi)\, d\xi}\to \infty$ as $t\to \infty$. This and inequality \eqref{eqa7} imply \eqref{eq12}. 
This completes the proof of Theorem \ref{thm2.3}. 
\end{proof}

Consider the following inequality:
\begin{equation}
\label{eqs5}
\frac{\beta(t)\nu^p(t)}{\alpha(t)} \le C < \bigg(\frac{1}{(p-1)\int_0^\infty \frac{\alpha(\xi)}{\nu^{p-1}(\xi)}d\xi}\bigg)^{\frac{p}{p-1}},\qquad t\ge 0,\quad C > 0,\quad p>1. 
\end{equation}
Let $\omega_0 := C^{\frac{1}{p}}$, i.e., $\omega_0^p = C$. Then it follows from the first inequality in \eqref{eqs5} that inequality \eqref{eqs6} holds for $\omega=\omega_0$. From the second inequality in \eqref{eqs5} and the relation $C = \omega_0^p$, one gets
$$
w_0^p < \bigg(\frac{1}{(p-1)\int_0^\infty \frac{\alpha(\xi)}{\nu^{p-1}(\xi)}d\xi}\bigg)^{\frac{p}{p-1}}.
$$
This implies
$$
\omega_0^{1-p} > (p-1)\int_0^\infty \frac{\alpha(\xi)}{\nu^{p-1}(\xi)}d\xi.
$$
Thus, if $g(0)>0$ is sufficiently small, then we have
$$
\big[g(0)+\omega_0\big]^{1-p} > (p-1)\int_0^\infty \frac{\alpha(\xi)}{\nu^{p-1}(\xi)}d\xi,\qquad p>1.
$$
Therefore, the function
$$
\frac{1}{(g(0) + \omega)^{1-p} - (p-1) \int_0^t \frac{\alpha(\xi)}{\nu^{p-1}(\xi)}\, d\xi}
$$
which appears in the right-hand side of \eqref{eqx20} 
is well-defined for all $t\ge 0$ when $\omega = \omega_0$ and $g(0)>0$ is sufficiently small. 
From the remarks above and the proof of Theorem \ref{thm2.3} we have the following corrollary
\begin{corrollary}
\label{corrollary2.5}
Assume that 
\begin{equation}
\label{}
\sup_{t\to\infty} \int_0^t \gamma(\xi)\, d\xi < \infty,
%\qquad  \int_0^\infty \alpha(t)\, dt <\infty,
\end{equation}
\begin{equation}
\label{}
\frac{\beta(t)\nu^p(t)}{\alpha(t)} \le C< C_1:=\bigg(\frac{1}{(p-1)\int_0^\infty \frac{\alpha(\xi)}{\nu^{p-1}(\xi)}d\xi}\bigg)^{\frac{p}{p-1}},\qquad 
\nu(t) = e^{-\int_0^t \gamma(\xi)\, d\xi}. 
\end{equation}
If $\|u_0\|$ is small such that
$$ 
\bigg(\|u_0\| + C^{\frac{1}{p}}\bigg)^p < C_1,
$$
then the solution $u(t)$ to problem \eqref{eq1} exists globally, is bounded, and satisfies
\begin{equation}
\|u(t)\| \le C_2 e^{\int_0^t \gamma(\xi)\,d\xi},\qquad t\ge 0,\qquad C_2 = const>0.
\end{equation}
In addition, if
\begin{equation*}
\label{}
\lim_{t\to\infty}\int_0^t \gamma(\xi)\, d\xi = -\infty,
%\qquad  \int_0^\infty \alpha(t)\, dt <\infty,
\end{equation*}
then 
\begin{equation}
\label{}
\lim_{t\to\infty} u(t) = 0.
\end{equation}
\end{corrollary}

\begin{theorem}
\label{thm2.4}
Assume that $g(0)=\|u(0)\|\not = 0$ 
and that $\alpha(t)\ge 0$ satisfies the inequality
\begin{equation}
\label{eqs9}
\alpha(t) \le \frac{(q-1)\beta(t)}{(q\zeta(t))^p},\qquad t\ge 0,\quad q>1,
\end{equation}
where
\begin{equation}
\label{eq22}
\zeta(t) := \frac{g(0)}{\nu(t)} + \frac{\int_0^t \beta(\xi)\nu(\xi)\, d\xi}{\nu(t)},\qquad \nu(t) = e^{-\int_0^t \gamma(\xi)\, d\xi}.  
\end{equation}
Then the solution $u(t)$ to problem \eqref{eq1} exists globally and 
\begin{equation}
\label{eq23bn}
\|u(t)\| < q\zeta(t),\qquad  \forall t\ge 0.
\end{equation}
In addition; 
\begin{enumerate}
\item[(a)]{
If the function $\zeta(t)$ is bounded on $[0,\infty)$, 
then the solution $u(t)$ to problem \eqref{eq1} is bounded.}
\item[(b)]{
If $\lim_{t\to\infty}\zeta(t) = 0$, 
then 
\begin{equation}
\label{eq20}
\lim_{t\to\infty} u(t) = 0.
\end{equation}}
\end{enumerate}
\end{theorem}

\begin{proof}
Recall from our earlier assumptions that $\alpha(t),\beta(t)$, and $\gamma(t)$ are in $L^1_{loc}([0,\infty))$ and $\alpha(t)\ge 0$, $\beta(t)\ge 0$, $t\ge 0$. Thus, the integrals 
$\int_0^t \gamma(\xi)\, d\xi$ and $\int_0^t \beta(\xi)\nu(\xi)\, d\xi$ are well-defined for all $t\ge 0$ and $\nu(t)>0$, $\forall t\ge 0$. Therefore, the function $\zeta(t)$ is well-defined on $[0,\infty)$. 

{\it Let us prove that the solution $u(t)$ to problem \eqref{eq1} exists globally.} 
Assume the contrary that the maximal interval of existence of $u(t)$ is $[0,T)$ where $0<T<\infty$. 
Let us first prove that
\begin{equation}
\label{eq23}
g(t) = \|u(t)\| < q\zeta(t),\qquad 0\le t< T.
\end{equation}
Since $\nu(0) = 1$, it follows from \eqref{eq22} with $t=0$ that $g(0)= \zeta(0)<q\zeta(0)$. This and the continuity of $g(t)$ and $\zeta(t)$ imply that there exists $\theta>0$ such that $g(t)< q\zeta(t)$, $\forall t\in [0, \theta]$. 
Let $T_1\in (0,T]$ be the largest real number such that
\begin{equation}
\label{eq24}
g(t) < q\zeta(t),\qquad \forall t \in [0,T_1).
\end{equation}
Let us prove that $T_1= T$. Assume the contrary. Then $0<T_1<T$. From the continuity of $g(t)$ and the definition of $T_1$, one has 
\begin{equation}
\label{eqa9}
g(T_1) = q\zeta(T_1),\qquad g(t)<q\zeta(t),\qquad  0\le t<T_1. 
\end{equation}

Inequalities \eqref{eq2}, \eqref{eqs9}, and \eqref{eq24} imply
\begin{equation}
\label{}
\begin{split}
\dot{g} &\le \gamma(t)g(t) + \alpha(t)(q\zeta(t))^p + \beta(t) \\
&\le \gamma(t)g(t) + (q-1)\beta(t) +\beta(t) = \gamma(t)g(t) + q\beta(t),\qquad \forall t\in [0,T_1].
\end{split}
\end{equation}
This implies
\begin{equation}
\label{}
\frac{d}{dt} (g(t)\nu(t)) \le  q\beta(t)\nu(t),\quad t\in [0,T_1],\qquad \nu(t) = e^{-\int_0^t \gamma(\xi)\, d\xi}.
\end{equation}
Integrate this inequality from 0 to $t$ to get
\begin{equation}
\label{}
g(t)\nu(t) - g(0)\nu(0) \le q\int_0^t \beta(\xi)\nu(\xi)\, d\xi,\qquad t\in (0,T_1].
\end{equation}
Thus,
\begin{equation}
\label{eq28}
\begin{split}
g(t) \le \frac{g(0)}{\nu(t)} + \frac{q\int_0^t \beta(\xi)\nu(\xi)\, d\xi}{\nu(t)} 
< q\bigg(\frac{g(0)}{\nu(t)} +\frac{\int_0^t \beta(\xi)\nu(\xi)\, d\xi}{\nu(t)} \bigg) = q\zeta(t),\qquad t\in (0,T_1].
\end{split}
\end{equation}
Inequality \eqref{eq28} for $t=T_1$ is $g(T_1)<q\zeta(T_1)$ which contradicts to the first equality in \eqref{eqa9}. This contradiction implies that $T_1=T$, i.e., inequality \eqref{eq23} holds. 

Inequality \eqref{eq23} implies that $\|u(t)\|$ is finite on the interval $[0,T]$.
Thus, by using Assumption A with $t_0=T$ one can extend the solution $u(t)$ to a large interval. In other words, there exists $\delta>0$ so that the solution $u(t)$ 
to equation \eqref{eq1} exists on $[0,T+\delta]$. This contradicts the definition of $T$. The contradiction implies that $T=\infty$, i.e., the solution $u(t)$ to equation \eqref{eq1} exists globally. 

Inequality \eqref{eq23bn} follows from inequality \eqref{eq23} when $T=\infty$. 
It follows directly from inequality \eqref{eq23bn} that if $\zeta(t)$ is bounded on $[0,\infty)$, then the solution $u(t)$ to equation \eqref{eq1} is bounded and that if $\lim_{t\to\infty}\zeta(t) = 0$, then $\lim_{t\to\infty}u(t) = 0$. 
 Theorem \ref{thm2.4} is proved.  
\end{proof}

A consequence of Theorem \ref{thm2.4} is the following result.

\begin{theorem}
\label{thm2.7}
Assume that $g(0)=\|u(0)\|\not = 0$
and that $\alpha(t)\ge 0$ satisfies the inequality
\begin{equation}
\label{eqs9s}
\alpha(t) \le \frac{(q-1)\beta(t)}{(q\zeta(t))^p},\qquad t\ge 0,\quad q>1,
\end{equation}
where
\begin{equation}
\label{eq22s}
\zeta(t) = \frac{g(0)}{\nu(t)} + \frac{\int_0^t \beta(\xi)\nu(\xi)\, d\xi}{\nu(t)},\qquad \nu(t) = e^{-\int_0^t \gamma(\xi)\, d\xi}.  
\end{equation}
Then the solution $u(t)$ to problem \eqref{eq1} exists globally.

In addition; 
\begin{enumerate}
\item[(a)]{
If
\begin{equation}
\label{eqs30s}
L:=\sup_{t\ge 0}\bigg|\int_0^t \gamma(\xi)\, d\xi\bigg| < \infty,  \qquad  \int_0^\infty \beta(t)\, dt <\infty,
\end{equation}
then the solution $u(t)$ to problem \eqref{eq1} is bounded.}
\item[(b)]{
If 
\begin{equation}
\label{eqs4s}
\lim_{t\to\infty}\int_0^t \gamma(\xi)\, d\xi = -\infty,\qquad \lim_{t\to\infty}\frac{\beta(t)}{\gamma(t)} = 0,
\end{equation}
then 
\begin{equation}
\label{eq20s}
\lim_{t\to\infty} u(t) = 0.
\end{equation}}
\end{enumerate}
\end{theorem}

\begin{proof}
{\it Let us proved that $u(t)$ is bounded assuming that inequality \eqref{eqs30s} holds.} 
From Theorem \ref{thm2.4}, it suffices to show that the function $\zeta(t)$ is bounded.
From the first inequality in \eqref{eqs30s}, one gets
\begin{equation}
\label{eqs35}
e^{-L}\le \nu(t) = e^{-\int_0^t \gamma(\xi)\, d\xi} \le e^{L},\qquad t\ge 0.
\end{equation}
Thus,
\begin{equation}
\label{eqs36}
0\le \int_0^\infty \beta(\xi)\nu(\xi)\, d\xi \le e^L \int_0^\infty \beta(\xi)\, d\xi < \infty.
\end{equation}
It follows from \eqref{eqs35} and \eqref{eqs36} that 
$$
\frac{g(0)}{\nu(t)} + \frac{\int_0^t \beta(\xi)\nu(\xi)\, d\xi}{\nu(t)} \le g(0)e^L + e^{2L} \int_0^\infty \beta(\xi)\, d\xi < \infty,\qquad \forall t\ge 0.
$$
Therefore, the function $\zeta(t)$ defined in \eqref{eq22} is bounded. 
Thus, $u(t)$ is bounded as a consequence of inequality \eqref{eq23bn}.

{\it Let us prove \eqref{eq20s} given that inequality \eqref{eqs4s} holds.} 
It follows from the first relation in \eqref{eqs4s} that $\lim_{t\to\infty} \nu(t) = \lim_{t\to\infty} e^{-\int_0^t \gamma(\xi)\, d\xi} = \infty$.  
We claim that 
\begin{equation}
\label{eq21}
\lim_{t\to \infty} \frac{\int_0^t \beta(\xi)\nu(\xi)\, d\xi}{\nu(t)} = 0.  
\end{equation}
Indeed, if $\int_0^\infty \beta(\xi)\nu(\xi)\, d\xi < \infty$, then $\int_0^t \beta(\xi)\nu(\xi)\, d\xi$ is bounded on $\mathbb{R}_+$ and relation \eqref{eq21} follows from the relation $\lim_{t\to\infty} \nu(t) = \infty$. If $\int_0^\infty \beta(\xi)\nu(\xi)\, d\xi = \infty$, then relation \eqref{eq21} follows from L'Hospital's rule 
and the relation $\lim_{t\to\infty}\frac{\beta(t)}{\gamma(t)} = 0$ (cf. \eqref{eqs4s}). 
%It follows from \eqref{eq21} and the relation $0<\nu(t)\to \infty$ as $t\to\infty$ that the supremum in %\eqref{eq22} is attained at a finite value of $t$. 

From \eqref{eq28} one gets
\begin{equation}
\label{eqa10}
0 \le g(t)\le \frac{g(0)}{\nu(t)} +\frac{q\int_0^t \beta(\xi)\nu(\xi)\, d\xi}{\nu(t)}, \qquad \forall t\ge 0.
\end{equation}
Relation \eqref{eq21}, the relation $\lim_{t\to\infty}\nu(t) = \infty$, and inequality \eqref{eqa10} imply that
$\lim_{t\to\infty} g(t) = 0$. Thus, relation \eqref{eq20s} holds. Theorem \ref{thm2.7} is proved. 
\end{proof}

\section{Acknowledgement}
The author thanks the referee for useful comments which improve the paper.

\end{document}